\documentclass[11pt]{amsart}
\usepackage{amsmath,amsthm,amsfonts,amscd,amssymb,eucal,latexsym, fullpage}
\usepackage{mathrsfs}
\usepackage[numbers]{natbib}
\usepackage[fit]{truncate}
\usepackage{fullpage}
\usepackage{graphicx}
\usepackage{epsfig}
\usepackage{color,soul}
\usepackage[all,cmtip]{xy}
\usepackage[top=1.125in, bottom=1.125in, left=1.25in, right=1.25in]{geometry}
\newcommand{\truncateit}[1]{\truncate{0.8\textwidth}{#1}}
\newcommand{\scititle}[1]{\title[\truncateit{#1}]{#1}}

\theoremstyle{plain}
\newtheorem{theorem}{Theorem}[section]

\newtheorem{lemma}[theorem]{Lemma}

\newtheorem{proposition}[theorem]{Proposition}

\theoremstyle{definition}
\newtheorem{definition}[theorem]{Definition}
\newtheorem{remark}[theorem]{Remark}

\def\de{\delta} \def\dl{\partial}   
\def\ep{\varepsilon}        \def\su{\subset}
    
\def\ra{\rightarrow}     
    \def\ti{\tilde}
\def\z{\times}  \def\x{\cdot}   
\def\p{\pitchfork}

\newcommand{\R}{\mathbb{R}}

\newcommand{\Z}{\mathbb{Z}}

\newcommand{\Mt}{M_\triangle}

\DeclareMathOperator{\area}{area}
\DeclareMathOperator{\mass}{mass}
\DeclareMathOperator{\diam}{diam}
\DeclareMathOperator{\dist}{dist}

\DeclareMathOperator{\HF}{HF}
\DeclareMathOperator{\SA}{SA}
\DeclareMathOperator{\vol}{vol}

\DeclareMathOperator{\Area}{Area}

\DeclareMathOperator{\Int}{Int}
\DeclareMathOperator{\supp}{supp}

\begin{document}
\begin{abstract}
P. Papasoglu asked in \cite{pp13} whether for any Riemannian
3-disk $M$ with diameter $d$, boundary area $A$ and volume $V$, there
exists a homotopy $S_t$ contracting the boundary to a point so that the
area of $S_t$ is bounded by $f(d,A,V)$ for some function $f$. He further
asks whether it is possible to subdivide $M$ by a disk $D$ into two
regions of volume  $V/4$ so that the area of $D$ is bounded by some function $h(d,A,V)$.

In this paper, we answer the questions above in the negative.
We further prove that given $N>0$ and $c\in(0,1)$, one can construct a metric $g'$ so that any 2-disk $D$ subdividing $(M,g')$ into two regions of volume at least $cV$, the area of $D$ is greater than $N$.
We also prove that for any Riemannian 3-sphere $M$, there is a surface that subdivides the disk into two regions of volume no less than $V/6$, and the area of this surface is bounded by $3\HF_1(2d)$, where $\HF_1$ is the homological filling function of $M$.
\end{abstract}

\scititle{Subdividing three-dimensional Riemannian disks}
\author{Parker Glynn-Adey}

\address{
Department of Mathematics, University of Toronto, 40 St. George Street.\\
Toronto, Ontario, M5S 2E4,\\
Canada\\
parker.glynn.adey@utoronto.ca}

\author{Zhifei Zhu}

\address{
Department of Mathematics, University of Toronto, 40 St. George Street.\\
Toronto, Ontario, M5S 2E4,\\
Canada\\
zhifei.zhu@mail.utoronto.ca}

\maketitle

\section{Introduction}	

In this paper we prove the following results.
\begin{theorem}\label{thm2}
Let $M$ be diffeomorophic to a 3-disk with boundary $\dl M\cong S^2$. For any
number $N>0$ there exists a Riemannian metric $g=g(N)$ on $M$ such that:
	\begin{itemize}
		\item The diameter of $M$ satisfies $\diam(M,g) \leq 10$.
        	\item The surface area of the boundary satisfies $\vol_2(\dl M,g)=4\pi.$
		\item The volume of $M$ satisfies $\vol_3(M,g)\leq 10$.
		\item For any homotopy $F:S^2\z [0,1] \ra M$ between $\dl M$ and a point $p\in M$ there exists some $t_0\in [0,1]$ such that the area $\vol_2(F(S^2,t_0))$ is greater than $N$.
	\end{itemize}
\end{theorem}

We will prove Theorem \ref{thm2} by explicitly constructing metrics on the 3-disk.
We will use a similar construction to prove the following:

\begin{proposition}\label{cor1}
	For any $N>0$ and $c\in(0,1)$ there exists a metric $g' = g'(c,N)$ on the 3-disk $M$ such that the diameter $d$, volume $V$ of $M$, and surface area $A$ of $\dl M$ are bounded above by 10, but any smooth embedded disc $D$ subdividing $M$ into two regions of volume greater than $cV$ has area greater than $N$.
\end{proposition}

\begin{remark}
	Panos Papasoglu and Eric Swenson independently constructed a Riemannian 3-sphere which is hard to cut by a surface of any genus in~\cite{Papasoglu2015}.
	Their expander based construction implies Theorem~\ref{thm2} and Proposition~\ref{cor1}; we give an alternative elementary construction for genus zero case.
\end{remark}

\begin{remark}
Theorem \ref{thm2} and Proposition \ref{cor1} do not imply each other.
In fact, in Theorem \ref{thm2}, the image $F(S^2,t)$ can be an immersed sphere which does not necessarily subdivide the 3-disk into two regions.
\end{remark}

We also show that given a Riemannian 3-sphere with diameter $d$ and volume
$V$, there exists a surface that subdivides the 3-sphere into two regions of
volume $>\frac{1}{6}V$, and the area of this surface is bounded in
terms of $d$ and the first homological filling function of the metric,
which is defined below.

\begin{definition}\label{defn: bisecting}
	Given a Riemannian 3-sphere $M$ with diameter $d$ and volume $V$, let
	\[ \SA(M)=\inf_{H \subset M} \{\vol_2(H): M\setminus H= X_1\sqcup X_2, \vol_3(X_i)>
	\frac{1}{6}V \text{ for }i=1,2\}\]
	be the subdivision area of $M$, where the infimum is taken over all embedded surfaces.
	Now we define
	\[\HF_1(\ell)=\sup_{||z||_1 \leq \ell}\left( \inf_{\partial c = z} \vol_2(c) \right) \]
	to be the first homological filling function.
	In the definition of $\HF_1(\ell)$ the supremum is taken over all 1-cycles $z$ satisfying $\vol_1(z) \leq \ell$ and the infimum computes the size of the smallest 2-cycle $c$ filling $z = \partial c$.
\end{definition}
\begin{theorem}\label{thm3}
	For any Riemannian 3-sphere $(M, g)$ with $M$ diffeomorphic to $S^3$ we have:
	\[\SA(M)\leq 3 \HF_1(2d)\]
	where $d$ is the diameter of $(M,g)$.
\end{theorem}

\begin{remark}
	One can replace the constant $\frac{1}{6}$ in Definition~\ref{defn: bisecting} by $\frac{1}{4}-\ep$ for any small $\ep>0$.
	The proof given below in Section~\ref{section3} is straightforward to adapt to this more general constant.
	We use a fixed constant only for the sake of concreteness.
\end{remark}

The proof techinque we use to show Theorem~\ref{thm3} is an adaptation of technique first developed by Gromov in \cite[\S1.2]{Gromov1983}.
The version of the technique that we employ was used by A. Nabutovsky and R. Rotman in \cite{NabutovskyRotman2006} to obtain the first curvature-free bounds on areas of minimal surfaces in Riemannian manifolds.

Theorem~\ref{thm2} and Proposition~\ref{cor1} answer the following questions by P. Papasoglu in \cite{pp13}:\\
Question 1:
	Let $M$ be a Riemannian manifold homeormophic to a 3-disk satisfying: (i) $\diam(M) = d$, (ii) $\area(\partial M) = A$, (iii) and $\vol_3(M)= V$.
	Is it true that there is a homotopy $S_t : \partial M \times [0,1] \rightarrow M$ such that:
	$S_0 = id_{\partial M}$ and $S_1$ is a point and $\vol_2(S_t) \leq f_1(A,d,V)$ for some function $f_1$?\\
Question 2:
	Let $M$ be as above. Is it true that there is a relative 2-disc $D$ splitting $M$ in to two regions of volume at least $V/4$ such that $\area(D) \leq f_2(A,d,V)$ for some function $f_2$?

These questions were inspired by the work of Ye.~Liokumovich, A. Nabutovsky, and R. Rotman in \cite{lio12}.
They proved that any Riemannian 2-sphere $(M,g)$ can be swept-out by curves of length
at most $200 \diam(M) \max\left\{1, \log\frac{\sqrt{\Area(M)}}{\diam(M)} \right\}$ and showed that this bound is optimal up to a constant factor.
Liokumovich, Nabutovsky, and Rotman's work was related to the work of S. Frankel and M. Katz~\cite{fra93}.
For further refinements of that work, see Liokumovich~\cite{liokumovich2013spheres} which constructs Riemannian 2-spheres which are hard to sweep out by 1-cycles.

In this work, we give negative answers to Question 1 (Theorem~\ref{thm2}) and Question 2 (Proposition~\ref{cor1}).
We do, however, prove a positive result which majorizes the size of the disk in Question 2 by the homological filling function and diameter of $M$ (Theorem~\ref{thm3}).

Papasoglu's Question 1 is a natural extension of the following question asked by Gromov in 1992: Consider all Riemannian metrics $(D^2, g)$ on the 2-disk such that: (i) the length of the boundary is at most one and (ii) the diameter of the disk is at most one. Is there a universal constant $C$ such that: For every such metric there is a free homotopy of curves which contracts the boundary to a point through curves of length at most $C$?

Frankel and Katz answered Gromov's question negatively in \cite{fra93}.
They construct a metric on the disc with a ``wall'' whose base is shaped like a regular binary tree with many nodes.
Combinatorial properties of the tree force any curve subdividing the nodes in to two equal parts to meet the edges of tree many times.
This curve will have to ``climb over the wall'' many times.
This  combinatorial obstruction forces any contraction of the boundary of the disc to a point to contain a long curve.
In our context we need to produce a large surface in any contraction of the boundary of $D^3$ to a point.
We do so by constructig a metric which is concentrated around two solid tori embedded in $D^3$ as a Hopf link.  The fact that the tori are linked forces any sweep-out of the 3-disk to meet one component of link transversally and hence any sweep-out will contain a 2-cycle of large area.

Lemma~\ref{lm3} tells us that any sweep-out of a 3-disk containing a pair of linked solid tori must contain an essential loop on the boundary of one of the tori. This essential loop will bound some disk in the solid torus. Our choice of the Riemannian metric on the solid tori forces any such disk to have a large area and this fact implies the assertion of Theorem~\ref{thm2} and Proposition~\ref{cor1}.

\begin{remark}
	Our construction of the metric above was inspired by D.~Burago and S.~Ivanov's construction of a metric on the 3-torus $(T^3,g)$ such that any homologically non-trivial 2-cycle in $T^3$ has large area~\cite{bur98}.
	L. Guth remarked that such a construction should provide an example of a sphere which is hard to bisect in~\cite{gut07}.
\end{remark}

\subsection{Outline} 
\label{sub:Outline}

 In Section \ref{section2}, we will first construct a Riemannian metric $g$ on the 3-disk and we will then prove that $g$ has the desired properties.
 Every contraction of the boundary of the disk will contain a surface of large area.
 In Section \ref{section3}, we give a method of decomposing Riemannian 3-sphere $(M,g)$ in to two regions with volume bounded below by a constant fraction of the total volume by a surface of area bounded by the homological filling function and diameter of the $M$.
 That is, we will prove Theorem~\ref{thm2} and Proposition~\ref{cor1} in Section \ref{section2} and Theorem~\ref{thm3} in Section \ref{section3}.

\section{Proof of Theorem \ref{thm2} and Proposition \ref{cor1}}\label{section2}

The constructions of the metrics in Theorem \ref{thm2} and Proposition
\ref{cor1} are similar. Let $(D^3,g_\circ)$ be a Euclidean $3$-disk of radius
one in $\R^3$, where $g_\circ$ is the standard Euclidean metric on $D^3$.
Consider a smooth embedding of a pair of solid tori $f:T_1 \sqcup T_2 \ra D^3$
into $D^3$ which embed the solid tori as a thick Hopf link (See Figure \ref{fig1}). We
modify the metric on the linked pair in the following way.

\begin{figure}[htbp]
\centering\includegraphics[width=5cm]{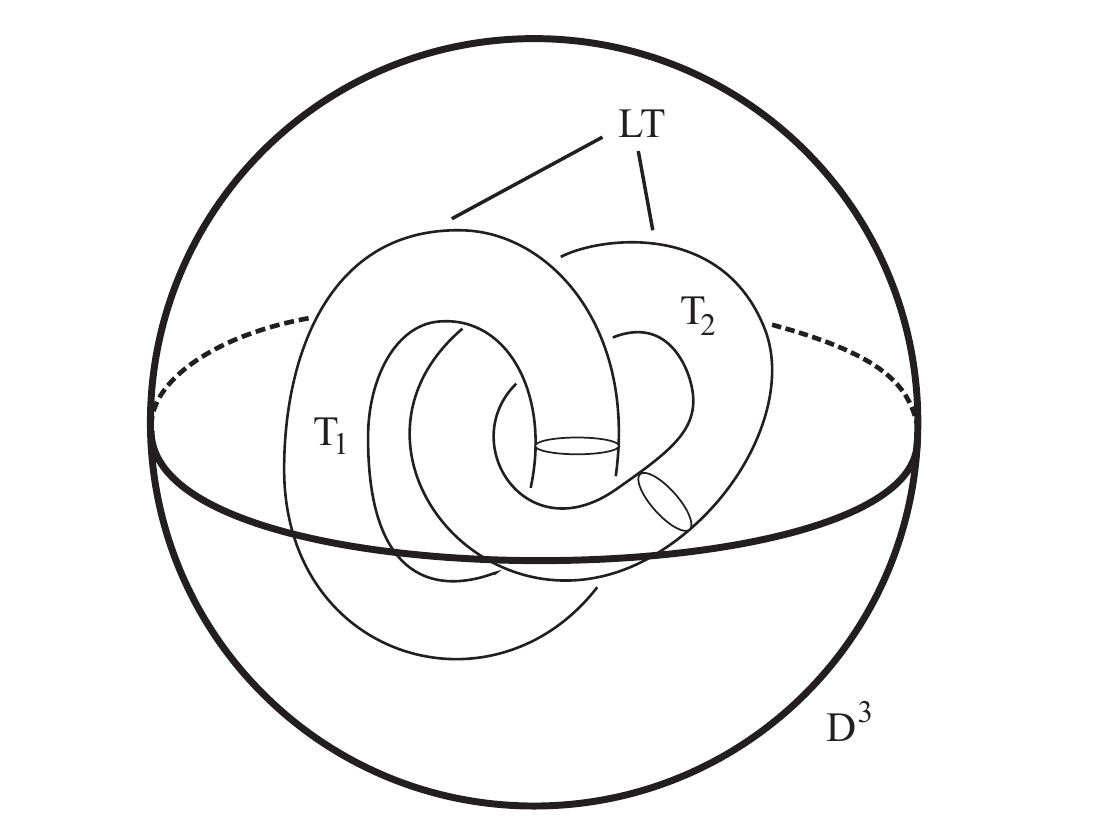}
\vspace*{8pt}
\caption{Embedding of linked solid tori.}\label{fig1}
\end{figure}

Let $D$ be a hyperbolic 2-disk of unit radius and curvature $-c^2$, where $c$ is
positive. The area $\vol_2(D)$ of the disk is proportional to
$\sinh(c)/c$. When $c\ra \infty$ we have $\vol_2(D)\ra \infty$.

Let $\de>0$ be a small constant. For any $N>0$, there exists an $\ep>0$ such
that $\vol_2(D)>N$ and $\ep \x \vol_2(D)<\de$. Let $C=D\z [0,\ep]$ be a
cylinder with the product metric obtained from the hyperbolic disc and the
interval. The area of a horizontal cross section of $C$ is $\vol_2(D)>N$, but
the volume $\vol_3(C)\leq \de$. Let $T$ be a solid torus obtained by
identifying the ends, $D\z\{0\}$ and $D\z\{\ep\}$, of $C$.
Let $\ti{g}$ be the induced metric on $T$. Note that the volume
$\vol_3(T,\ti{g})$ of the solid torus is bounded by $\de$.

We denote by $LT$ the image of the pair of linked solid tori in $D^3$. Let $U$
be an open neighbourhood of $LT$. Let us define the metric $\ti{g}$ on $LT$ to
be $\ti{g}$ as above and then smoothly extend it to the neighbourhood $U$. Let
$V$ be an open neighbourhood of $D^3\setminus U$ such that $V\cap
LT=\emptyset$. We take two smooth bump functions $f_1$ and $f_2$ on $U$ and $V$
respectively such that $f_1+f_2=1$ and define the metric $g$ on $D^3$ to be
$f_1\cdot\ti{g}+f_2\cdot g_\circ$. Note that for any $\de>0$, one can choose
$U$ and $V$ so that the volume $\vol_3(D^3,g)$ is bounded by
$\frac{4}{3}\pi+\de$. We claim that, with a proper choice of the constants, the
metric $g$ has the desired properties in Theorem~$\ref{thm2}$ and
Proposition~\ref{cor1}.

We will first prove Theorem \ref{thm2}. Let $F: \partial D^3 \z [0,1]\ra D^3$ a homotopy
between the boundary $\dl D^3$ and a point $p\in D^3$. By following lemma, it suffices to prove Theorem~\ref{thm2} in the case where the intersection of $F(S^2,t)$ and $\dl LT$ is transverse for all but finitely many $t\in [0,1]$.

\begin{lemma}\label{lm1}
	Let $X$ be an embedded submanifold of $D^3$.
	We write $S^2 = \partial D^3$.
	For any homotopy $F:S^2\z [0,1]\ra D^3$ of the boundary $\dl D^3$ to a point, there exists a smooth homotopy $\widetilde{F}:S^2\z [0,1]\ra D^3$, such that $\max_{t\in [0,1]}|\vol_2(\widetilde{F}(S^2,t))-\vol_2(F(S^2,t))|<\ep$ and $\widetilde{F}(S^2,t)$ is transverse to $X$ for all but finitely many $t\in [0,1]$.
\end{lemma}

\begin{proof}
	We will use Thom Multijet Transversality Theorem (\cite{GG12}, Theorem 4.13) to prove this Lemma. Let $\iota:X\hookrightarrow D^3$ be the embedding of $X$. We denote the map $F(\cdot,t):S^2\z\{t\}\ra D^3$ by $F_t$. Consider the map $\pi:F^{-1}(\iota(X))\ra [0,1]$, which is the restriction of the projection $X\z[0,1]\ra[0,1]$ to the domain $F^{-1}(\iota(X))$. Note that for $t\in[0,1]$, the intersection $F_t\p \iota$ is transverse if and only if $t$ is a regular value of the projection $\pi$. The idea is that because $[0,1]$ is a one-dimensional manifold, we can perturb $\pi$ to be a Morse function with distinct critical values.

 Indeed, by Proposition 6.13 and Theorem 4.13 in \cite{GG12}, the set of Morse functions all of whose critical values are distinct form a residual set in $C^{\infty}(F^{-1}(\iota(X)),[0,1])$. Therefore, given any homotopy $F:S^2\z[0,1]\ra D^3$, we first take a sequence of smooth homotopies $F_j$ approximating $F$. Then for each $F_j$, we take a sequence of maps $F_{jk}$ approximating $F_j$ such that each $F_{jk}$ satisfies $\pi:F_{jk}^{-1}(\iota (X))\ra [0,1]$ is Morse. And finally, we take $\ti{F}$ to be $F_{kk}$ for $k$ sufficiently large.
\end{proof}

We will work with basic notions of geometric measure theory encompassed by the following definitions from \cite{federer1969geometric}.
The notion of flat norm was originally introduced by H. Whitney in \cite{whitney2012geometric} and was used to describe the cycle spaces of manifolds in W. Fleming's paper \cite{Fleming1966flat}.
We will use the following standard terminology:

\begin{definition}
	\label{defn: flat k-chains}
	For a Riemannian manifold $M$, the space of Lipschitz $k$-chains in $M$ with coefficients in $G$, an abelian group, is $C_k(M; G) = \{ \sum a_if_i : f_i: \Delta^k \rightarrow M \text{ a Lipschitz map}, a_i \in G\}$. We write $Z_k = \ker \partial_{k}$ for the cycles in $C_k$, where $\dl_k:C_k\ra C_{k-1}$ is the boundary map.

 When $G = \mathbb{Z}$ or $\mathbb{Z}_2$ and $c = \sum a_i f_i \in C_k(M; G)$ we define the $k$-mass of $c$ to be:
 	\[\mass_k(c) = \sum_{i} |a_i| \int_{\Delta^k} f_i^*(d \vol_g)\]

	We endow the space of integral cycles with the flat norm given by:
	\[ ||c|| = \inf_{d \in C_{k+1}} \mass_k(c + \partial d) + \mass_{k+1}(d) \]
	We define the flat distance between two cycles $c_1,c_2$ to be the flat norm $||c_1-c_2||$.
\end{definition}

We give a definition of sweep-out which we first encountered in~\cite{gut07}.

\begin{definition}
	\label{defn: sweep-out}
	Let $M$ be an $n$-dimensional manifold and $G$ an abelian group.
	Denote the space of integral flat cycle with coefficient group $G$ on $M$ by $Z_*(M;G)$.
	A $(n-k)$-dimensional family of $k$-cycles $z$ is a continuous map from an $(n-k)$-simplicial complex $K$ to $Z_k(M;G)$.
	We say $z$ is a sweep-out of $M$, if $z$ induces a nontrivial gluing homomorphism.
	The gluing homomorphism is constructed in the following way.

	We pick a fine triangulation of $K$ so that if $v_j$ and $v_{j+1}$ are two neighbouring vertices in the triangulation of $K$ then the $k$-cycle $(z(v_j)-z(v_{j+1}))$ bounds a $(k+1)$-chain in $M$.

	We map the edge $[v_j,v_{j+1}]$ to this $(k+1)$-chain.
	Inductively, one can map each $i$-simplex $\Delta^i$ of $K$ to a $(k+i)$-chain in $M$ as long as the triangulation is sufficiently fine.
	This induces a chain map between the simplicial complex of $K$ and singular complex of $M$.
	This chain map induces, on homology, the gluing homomorphism $z_{*}:H_{n-k}(K;\Z)\ra H_n(M;\Z)$.
	One can check that $z_{*}$ is independent of choice of the chains if the triangulation is fine enough and the fillings are chosen appropriately.
\end{definition}

In this work we will only use sweep-outs of $T^2$ by families of 1-cycles.
For more detailed discussions of sweep-outs, one may refer to \cite[Section 1]{gut07}, \cite[Section 2]{BS2010} and \cite{al1962}.


Note that the homotopy $F$ induces a degree one map $\ti{F}: (S^2\z[0,1])/(S^2\z\{0,1\}) \ra D^3/\dl D^3$.
By restricting the map $F$ to $F^{-1}(\dl LT)$, the transversality of $F(S^2,t)\p \dl LT$ implies that the one parameter family of 2-cycles $F(S^2,t)$ induces a sweep-out of the embedded pair of tori $\dl LT$ by 1-cycles. Indeed, for $t\in [0,1]$, $F(S^2,t)\p \dl LT$ defines a continuous family of 1-cycles.
When we consider the image of this family of 1-cycles under the gluing homomorphism, we will obtain precisely the degree one map $\ti{F}$.
We will prove that these 1-cycles contain a meridian of one of the tori in $\dl LT$, which bounds a disk with large area.
To begin with, we first show in Lemma \ref{lm3} that the family of 1-cycles must contain at least one non-contractible loop on the torus.

Let $T^2$ be Riemannian 2-torus and $z:S^1 \ra Z_1(T^2;\Z)$ a one-parameter family of 1-cycles based at a constant loop which is a sweep-out of $T^2$. Let us parameterize $S^1$ by $[0,1]$ and define $z_t = z(t)$ so that $z_0=z_1$ is the constant loop. Note that for every $t$, the 1-cycle $z_t$ is obtained from the the intersection between a 2-cycle and a torus, thus we may represent $z_t$ as a finite sum $z_t=\sum_{i=1}^{n(t)} a_{t,i} f_{t,i}$, where each $a_{t,i} \in \Z$ and $f_{t,i}:S^1\ra T^2$. We show the following:

\begin{lemma}\label{lm3}
	Let $T^2$ be Riemannian 2-torus and $z_t=\sum_{i=1}^{n(t)} a_{t,i} f_{t,i}$ be as above.
	If the family of cycles $z_t$ is a sweep-out of $T^2$ then there exists some $t_0\in[0,1]$ and $1\leq i_0\leq n(t_0)$, such that the image $f_{t_0,i_0}(S^1)$ is a non-contractible loop in $T^2$.
\end{lemma}

\begin{proof}
	We proceed by contradiction. Suppose for all $t$, every component $f_{t,i}$ of $z_t$ is contractible.
	We show that $z$ induces a trivial gluing homomorphism $z_*:H_1(S^1;\Z)\ra H_2(T^2;\Z)$.

	More precisely, the family $z_t$ taken as a map $z : S^1 \ra Z_1(T^2; \Z)$, represents a class $[z]$ in the fundamental group $\pi_1(Z_1(T^2;\Z))$.
	We would like to represent the class $[z]$ by a sum $\sum_{i=1}^N c_i$, where each $c_i:S^1\ra Z_1(T^2;\Z)$ is a continuous family of circles in $T^2$.
	If every loop is contractible, we show that every $c_i$ corresponds to a spherical class in $H_2(T^2;\Z)\cong\pi_1(Z_1(T^2;\Z))$ which is a contradiction. See the Figure~\ref{fig10} for the visual intuition of the proof.

	The maps $c_i$ are obtained in the following way.
	We first construct a one-dimensional simplicial complex $P$ from $z$.
	Recall that for every $t\in[0,1]$ the cycle $z_t$ is a union of $k_t:=\sum_{i=1}^{n(t)} |a_{t,j}|$ circles.
	We say that the topology of $z_t$ changes if: (i) a circle $f_{t,j}$ vanishes or (ii) several circles merge; we define these terms below.

\begin{definition}
The set $\{f_{t,j}\}_{j\in J}$ of circles merge at $t=s$, if there is a map $z'_{s}:\bigvee_{i\in I}S_i^1\ra T^2$ such that the flat distance $\lim_{t\ra s}\left\| z'_{s}-\sum_{i\in I}f_{t,i}\right\|=0$ for some finite index set $I\su \{1,2,\dots,k_s\}$, where $\|\cdot\|$ is the flat norm defined in Definition~\ref{defn: flat k-chains}.
And we say that a particular circle $f_{t,i}$ vanishes if $\lim_{t\ra s}\left\| z'_{s}-f_{t,i}\right\|=0$ where $z'_{s}$ is a constant map.
\end{definition}

Note that by Lemma~\ref{lm1}, we may perturb the family $\{z_t\}$ so that: for every $t$, the topology of $z_t$ changes at most once at $t$. We define the $0$-skeleton of $P$ as follows. For $t=0$ and $1$, we define one vertex $v_0$ (resp. $v_1$) that corresponds to the constant loops $z_0$ (resp. $z_1)$.
	We create a vertex $v_t$, if for any sufficiently small $\ep>0$, we have $|k_t-k_{t+\ep}|\neq 0$.

	And the $1$-skeleton of $P$ is defined as follows. Let $v_{t_1}$ and
	$v_{t_2}$ be two consecutive vertices, where $t_1<t_2$. For $t\in
	(t_1,t_2)$, we have that $k_t$ is a constant and there are $k_t$
	circles. For each $f_{t,i}$ which changes topology at $t_1$ and $t_2$,
	we define an edge $e_i$ connecting $v_{t_1}$ and $v_{t_2}$. Let us
	parameterize the edge by $e_i:[t_1,t_2]\ra P$. (See Figure~\ref{fig8}.)
	Note that if $z_t$ is induced by a Morse function then $P$ is just the
	corresponding Reeb graph.

\begin{figure}[htbp]
\begin{minipage}[c]{0.45\textwidth}
\centering\includegraphics[width=5cm]{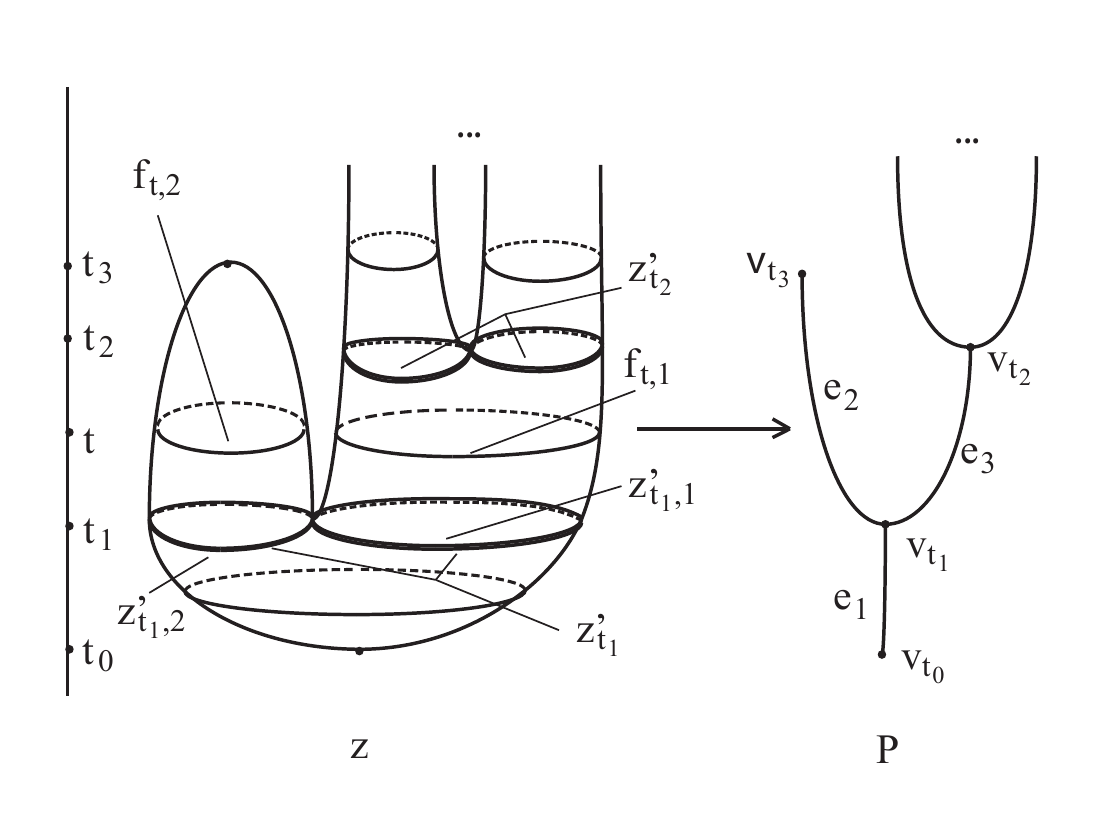}
\vspace*{8pt}
\caption{Construction of $P$. Here there are two merges at $t_1$ and $t_2$. And a circle vanishes at $t_3$.}\label{fig8}
\end{minipage}
\begin{minipage}[c]{0.45\textwidth}
\centering\includegraphics[width=5cm]{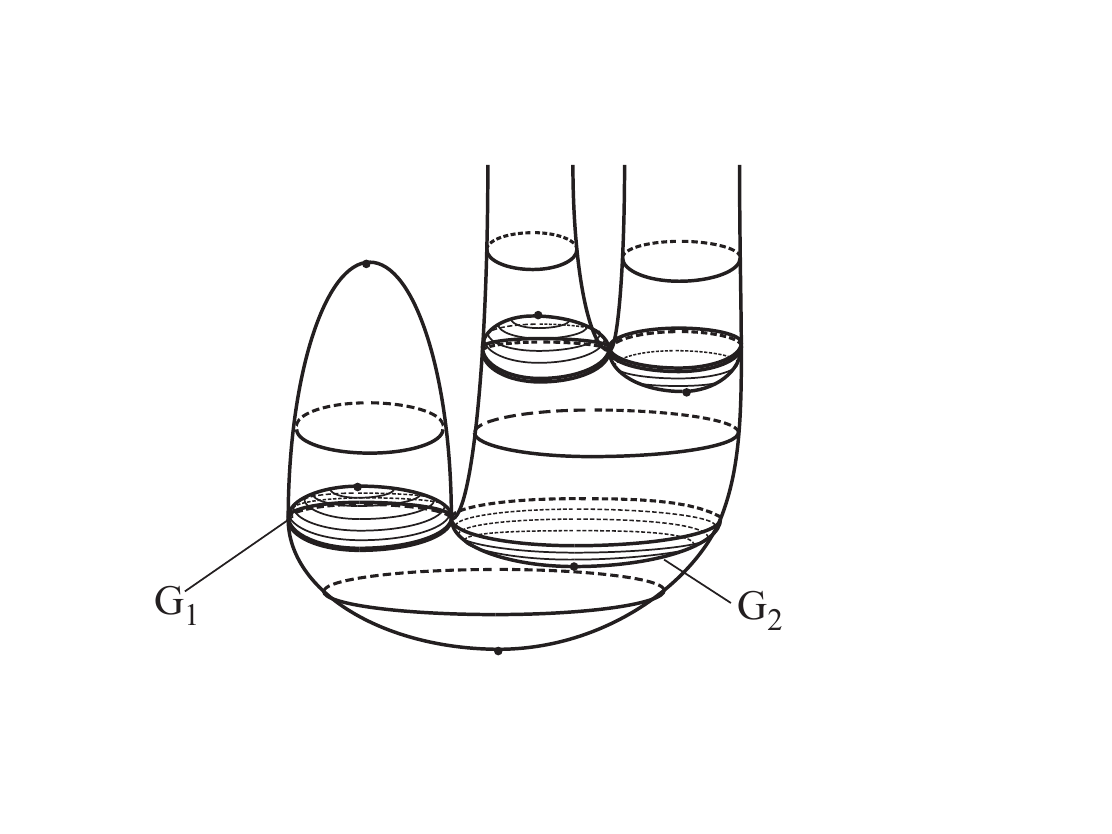}
\vspace*{8pt}
\caption{Contraction of the circles at a vertex.}\label{fig9}
\end{minipage}
\end{figure}

A vertex $v_{s}$ of our graph at $t=s$ either corresponds to a constant loop
or, for some index set $I\su \{1,2,\dots,k_s\}$, there is a map
$z'_{s}:\bigvee_{i\in I}S_i^1\ra T^2$ such that the flat distance $\lim_{t\ra
s}\| z'_{s}-\sum_{i\in I}f_{t,i}\|=0$.  In the latter case, let
$\iota_j:S_j^1\hookrightarrow  \bigvee_{i\in I}S_i^1$ be the natural inclusion
map for $j \in I$. By our assumption, every $z'_{s}(\iota_j(S^1_j))$ is contractible in $T^2$. We
parameterize this contraction of by $G_j:S^1\z [0,\de]\ra T^2$, where $\de$ is
a small positive number, $G_j(S^1\z \{0\})=z'_{s}(\iota_j(S^1_j))$ and $G_j(S^1\z
\{\de\})$ is a point in $T^2$. (See Figure \ref{fig9}.)

For each edge $e_i:[t_1,t_2]\ra P$, we will define the corresponding map $c_i:[0,1]\ra Z_1(T^2;\Z)$ in the following way.

If $t_1=0$ or $t_2=1$ then, by our assumption, the vertices $v_0$ and $v_1$ correspond to a single constant loop and we may define $c_i(t_1)$ or $c_i(t_2)$ to be this loop. For every $t\in (t_1,t_2)$, we will define $c_i(t)$ to be the cycle $f_{t,i}:S^1\ra T^2$.

Suppose $t_1 \neq 0$. When $t=t_1$, if $e_i(t_1)$ is a vertex corresponding to a constant loop\footnote{See the vertex $v_{t_3}$ in Figure~\ref{fig8}}, define $c_i(t_1)$ to be this loop. Otherwise, there is a merged circle $z'_{t_1}:\bigvee_{i\in I}S^1_i\ra T^2$ at $t_1$. Let $G: \left( \bigvee_{i\in I}S^1 \right)\z [0,\de]\ra T^2$ be a contraction such that $\|G(\iota_j(S^1)\z [0,\de])\|=\|G_j\|$, for every $j \in I$,  where $G_j$ is the map defined above. We note that this map $G$ keeps track of all the individual contractions of circles given to us by hypothesis. It is straightforward to check that such a $G$ exists since we can reparameterize each $G_j$ without changing the flat norm. Let $f:S^1\ra \bigvee_{i\in I}S^1$ be a degree one surjection, and now define $c_i|_{[t_1-\de,t_1]}(t)$ to be the cycle $G(f(\cdot),t_1-t) : S^1 \rightarrow T^2$.

\begin{figure}[htbp]
\centering\includegraphics[width=12cm]{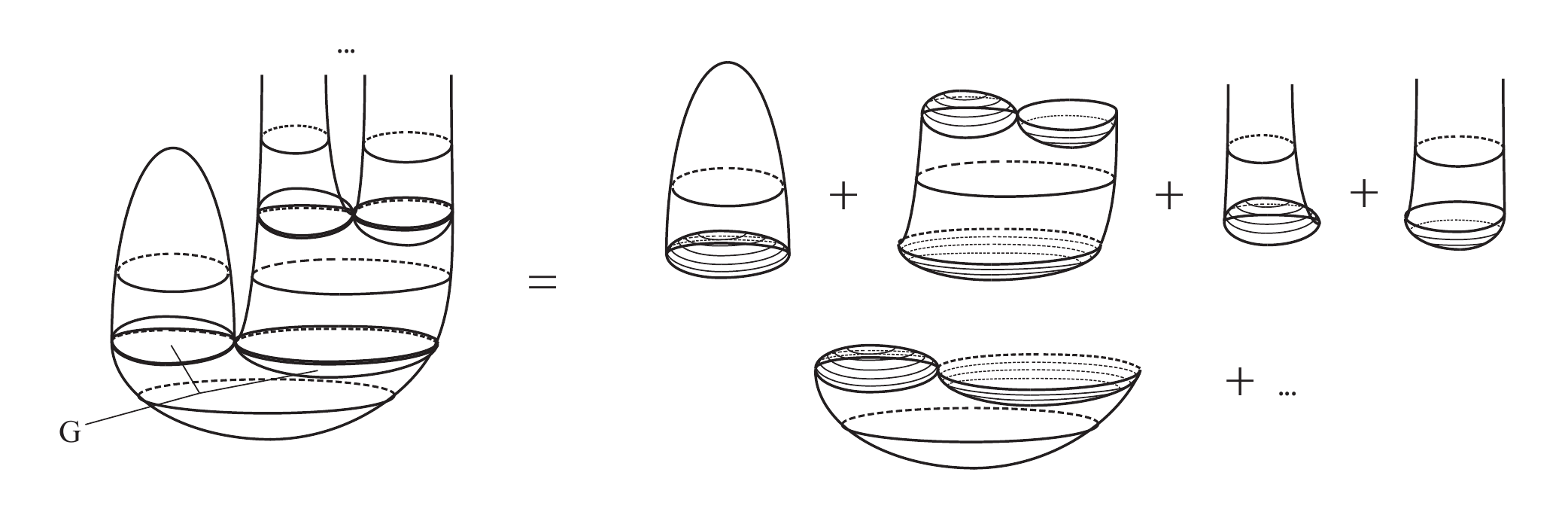}
\vspace*{8pt}
\caption{Representing the class $[z]$ by $\sum c_i$. We insert two copies of the disks obtain from $G$ at each vertex.}\label{fig10}
\end{figure}

Similarly, when $t_2\neq 1$, one can define $c_i|_{[t_2,t_2+\de]}(t)$ to be the cycle $G(f(\cdot),t-t_2) : S^1 \rightarrow T^2$ for the corresponding contraction $G$. This completes the construction.

Let $c=\sum_i c_i:[0,1]\ra Z_1(T^2; \Z)$, where the sum is over the set of all edges in $P$. (See Figure \ref{fig10}.) Observe that $c$ represents the same element as $z$ in $\pi_1(Z_1(T^2;\Z))$. To see this note that, up to a reparameterization, the image of $z-c$ corresponds to a union of spheres in $T^2$ obtained by gluing two copies of the contracting disks $G$ at each vertex. Because every spherical class is trivial in $H_2(T^2;\Z)$ the spheres represent the trivial element in $\pi_1(Z_1(T^2;\Z))$ under Almgren's isomorphism.

Finally, we show that each $c_i$ induces a trivial gluing homomorphism. Indeed, let $\eta>0$. We pick a fine subdivision of $S^1$ so that if $u_j$ and $u_{j+1}$ are two consecutive vertices in the subdivision then the flat distance between $c_i(u_j)$ and $c_i(u_{j+1})$ is less than $\eta$. We take $\eta$ to be sufficiently small so that $c_i$ induces a well-defined gluing homomorphism. See Definition~\ref{defn: sweep-out}.

Because each $c_i(u_j)$ is a contractible loop in $T^2$ one may homotope $c_i(u_j)$ to $c_i(u_{j+1})$ for consecutive vertices in the fine subdivision of $S^1$. Because $c_i(0)=c_i(1)$ is a constant loop one may contract $c_i(u_j)$ to the constant loop through $[0,u_j]$ and $[u_j,1]$. Therefore, we conclude that the image of the fundamental class $[S^1]$ under the gluing homomorphism $c_{i*}:H_1(S^1)\ra H_2(T^2)$ can be represented by a map from the sphere $S^2\ra T^2$, where the $S^2$ is obtained by gluing the two contractions of $c_i(u_j)$. This shows that $c_{i*}$ is trivial and further implies $z_*$ is trivial, which is a contradiction.
\end{proof}

We now use the linking property of the solid tori $LT$ to show that these 1-cycles indeed contain a meridian circle of $\dl LT$. Recall that $T_i$ is the connected component of $LT$ and $C_i$ is the cylinder $D\z [0,\ep]$, where $D$ is a hyperbolic 2-disk.

\begin{definition}
 Let us identify the solid tori $T_i$ with $C_i/ \left( (x,0)\sim (x,\ep) \right)$, $x\in D$. Let $F:S^2\z I\ra D^3$ be a homotopy contracting the boundary $\dl D^3$ to a point. We say the immersed sphere $F(S^2,t)$ intersects $\dl T_i$ along the essential circle $\dl D$, if the projection map $p_D:F(S^2,t)\cap C_i\ra D \times \{0\}$ to the base of the cylinder $C_i$ is surjective. (See Figure \ref{fig4}.) In particular, the restriction $p_D:\dl (F(S^2,t)\cap C_i)\ra \dl D \times \{0\}$ is also surjective.
\end{definition}

\begin{figure}[htbp]
\centering\includegraphics[width=5cm]{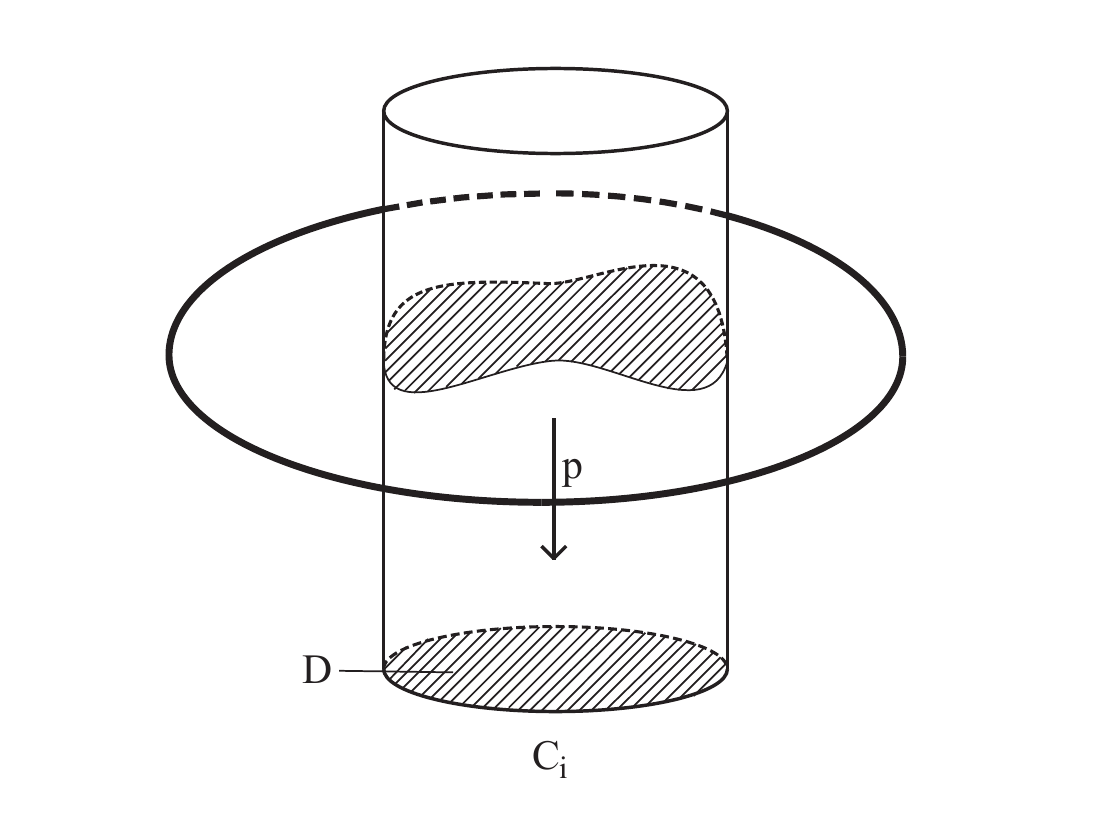}
\vspace*{8pt}
\caption{Intersect along the essential circle $\dl D$.}\label{fig4}
\end{figure}

\begin{lemma}\label{lm2}
Let $f:T_1\sqcup T_2\ra D^3$ be a topological embedding of a pair of solid tori into $D^3$ and such that the image $LT$ is a Hopf link. Let $F:S^2\z [0,1]\ra D^3$ be a homotopy contracting $\dl D^3\cong S^2$ to a point. Suppose that for almost all $t\in [0,1]$, the intersection $F(S^2,t)\p \dl LT$ is transverse. Let us identify each solid torus with a cylinder $C_i=D\z [0,\ep]$. Then there exists some $t_0\in [0,1]$ such that $F(S^2,t_0)$ intersects one of the $\dl T_1$ and $\dl T_2$ along the essential circle $\dl D$.
\end{lemma}

\begin{proof}
Let $F_t=F( \cdot ,t):S^2\ra D^3$. For almost all $t\in[0,1]$ we write $z_t =  F_t(S^2) \cap \partial LT$ as a union of circles $\cup_{k=1}^{n_t} S^1_{t,k}$ for some integer $n_t>0$.

Let $A_i=\dl D \times \{0\}$ in $\partial T_i$ be an essential circle along the meridian of the torus and $B_i$ an essential circle along the equator. We identify the torus $\dl T_i$ in the usual way with a square whose boundary is $A_i$ and $B_i$ and define $p_{A_i}:\dl T_i\ra A_i$ and $p_{B_i}:\dl T_i\ra B_i$ be the projection maps of this square to its boundary edges.

When we analyze the behaviour of the circles $S^1_{t,k}$ one of the following three cases must occur:
\begin{enumerate}
	\item The circles $S_{t,k}$ are contractible in $\dl LT$ for every $t$.
	\item For some $t_0\in[0,1]$, the circles $\{S_{t_0,k}\}_{k = 1}^{n_{t_0}} $ are not all contractible in $\dl LT$ but the map $p_{A_i}:S_{t_0,k}\ra A_i$ has degree 0 for every $k=1,2,\dots,n_{t_0}$ and $i=1,2$.
	\item For some $t_0$ and $k_0$ the circle $S_{t_0,k_0}$ is non-contractible in $\dl LT$ and the map $p_{A_i}:S_{t_0,k_0}\ra A_i$ has non-zero degree for $i=1$ or $2$.
\end{enumerate}

\begin{figure}[htbp]
\centering\includegraphics[width=5cm]{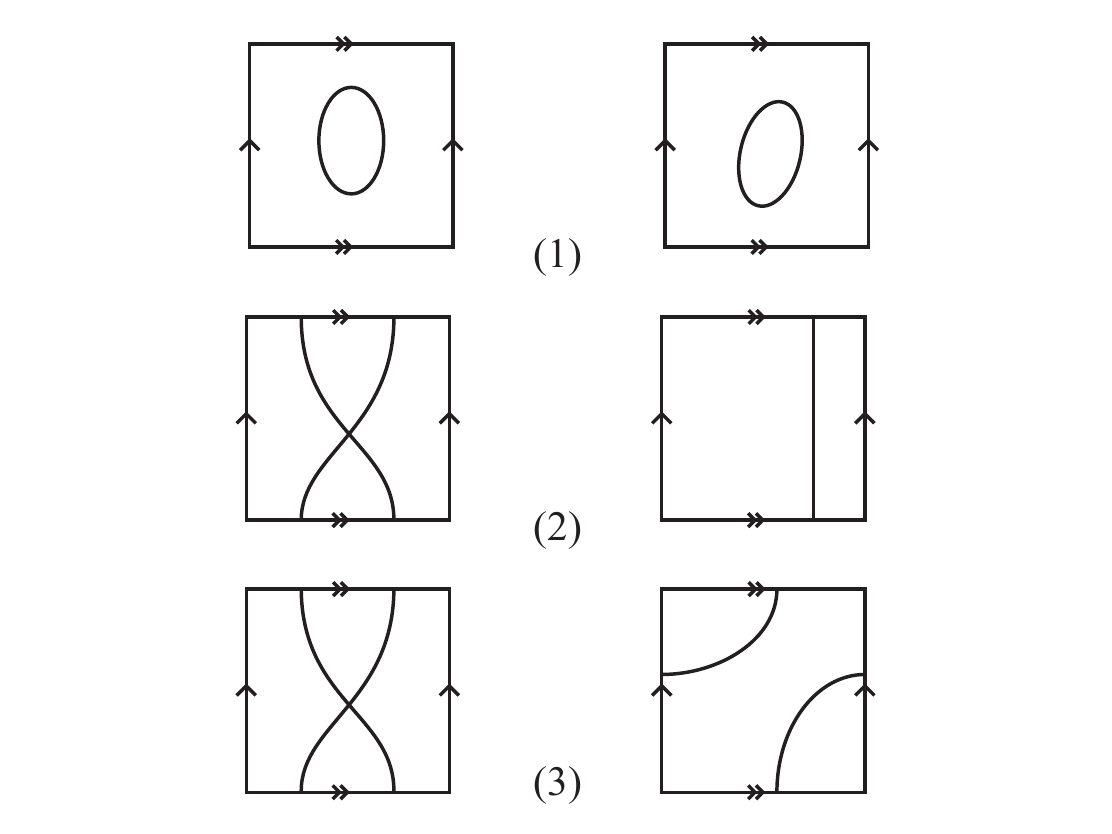}
\vspace*{8pt}
\caption{Examples of the three cases.}\label{fig5}
\end{figure}

Note that (1) is impossible because of Lemma \ref{lm3}. If $t_0$ is the time such that case (3) happens, then we are done. Indeed, without loss of generality, let us assume that $p_{A_1}:S_{t_0,k_0}\ra A_1$ has non-zero degree. Note that the circle $S_{t_0,k_0}$ bounds a disk $\Delta$ in $F_{t_0}(S^2)$ and the projection $p_{A_1}:S_{t_0,k_0}\ra A_1$ is the restriction of $p:\Delta\cap T_1\ra D$ to its boundary. We will check that $p$ is surjective. Consider the following commutative diagram.
\[\xymatrix{
  \dots  \ar[r] & H_2(\Delta) \ar[d] \ar[r] & H_2(\Delta/S_{t_0,k_0}) \ar[d]_{\ti{p}_*} \ar[r]^{\cong} & H_1(S_{t_0,k_0}) \ar[d]_{(p_{a,1})_*} \ar[r] & H_1(\Delta) \ar[d] \ar[r] & \dots \\
  \dots \ar[r] & H_2(D) \ar[r] & H_2(D/A_1) \ar[r]^{\cong} & H_1(A_1) \ar[r] & H_1(D) \ar[r] &\dots   }\]
The map $p$ is surjective since, by the long exact sequence above, the middle map $\ti{p}:\Delta/S_{t_0,k_0}\ra D/A_1$ induced by $p$ has non-zero degree.

Finally we show that case (2) is impossible. In case (2), there exists some non-contractible circle $S_{t_0,k_0}$ of $z_{t_0}$. Since the degrees of the maps $p_{A_1}$ and $p_{A_2}$ are both zero, the non-contractibility of $S_{t_0,k_0}$ in $\partial LT$ implies one of the maps $p_{B_1}$ and $p_{B_2}$ has non-zero degree. Let us assume $p_{B_1}$ has non-zero degree. This implies that the linking number between $S_{t_0,k_0}$ and $B_2$ is non-zero. Let $\Delta$ be a disk bounded by $S_{t_0,k_0}$ in $F_{t_0}(S^2)$.

Then the intersection $\Delta\cap T_2$ is non-empty. Because, otherwise, we may unlink $S_{t_0,k_0}$ with $B_2$ by contracting $S_{t_0,k_0}$ in $\Delta$. In fact, this further implies $B_2$ intersects some non-contractible circle $S_{t_0,k_1}$ in $\Delta\cap\dl T_2$. However, by our assumption, the map $p_{A_2}:S_{t_0,k_1}\ra A_2$ has degree 0. This is a contradiction, because in this case, by contracting its image under $p_{A_2}$, we may homotope $S_{t_0,k_1}$ such that the intersection between $S_{t_0,k_1}$ and $B_2$ is empty.

\end{proof}
We may now complete the proof of Theorem \ref{thm2}.
\begin{proof}[Proof of Theorem \ref{thm2}]
	Fix any $N > 0$.
	Our choice of $N$ fixes a specific hyperbolic disc $D$.
	We choose $\delta = \delta(N) > 0$ as above at the beginning of this section.
	Let $(D^3,g)$ be the Riemannian 3-disk we've constructed.
	Recall the linked tori $LT$ consisting of $T_1$ and $T_2$ in a Hopf link.
	See Figure~\ref{fig1}.
	Let $F:S^2\z [0,1]\ra D^3$ be a homotopy from $S^2\cong \dl D^3$ to a point $p\in D^3$
	By our construction, the volume $\vol_3(D^3,g)$ is bounded by $\de+\frac{4}{3}\pi\leq 10$.

	We claim that the diameter $d$ of $(D^3,g)$ is also bounded.
	Indeed, let $U$ be the tubular neighbourhood of $LT$ as above.
	By construction, for any point $x\in LT$, the distance $\dist(x,\dl LT)\leq 1$ and for
$y\in D^3\backslash U$ we have $\dist(y,\dl D^3)\leq 1$.
	Therefore, with an appropriate choice of $U$ and $\de$, we have $\dist(x,\dl D^3)\leq 2+\eta$, for some small $\eta>0$. And then the diameter $d\leq 2\x2+2\eta+\pi<10$, since every pair of points in $\dl D^3$ can be connected by a curve of length at most $\pi$.

	Finally, we show that there exists a $t_0\in [0,1]$ such that the area of the intersection of the image of $F$ with $LT$ satisfies $\vol_2(F_{t_0}(S^2)\cap LT)\geq N$.
	By Lemma \ref{lm2}, there is a $t_0\in [0,1]$ such that $F_{t_0}(S^2)$ intersects one of the tori $\dl T_1$ or $\dl T_2$
along the essential circle $\dl D$.
	Without loss of generality let us assume the projection map $p:F_{t_0}(S^2)\cap T_1\ra D$ is surjective.
	Because the metric $g$ on the cylinder $D\z [0,\ep]$ is a product metric, the projection map $p$ is area
decreasing. That is, $\vol_2(F_{t_0}(S^2)\cap T_1)\geq \vol_2(D)\geq N$, which completes the proof.
\end{proof}

Using a similar argument, we are able to prove Proposition
\ref{cor1}. Recall that in Proposition \ref{cor1}, we would like to show
that given any large number $N$ and a fixed $c\in(0,1)$, there exists a
metric $g'$ on the 3-disk $M$ with bounded diameter, volume, and boundary
area. And for any smooth boundary relative embedded disk $f:(D^2,S^1)\ra (M,\dl M)$ that subdivides
$M$ into two regions of volume at least $c\x\vol_3(M)$ we have that the area of the disk $f(D^2)$ is greater than $N$.

\begin{proof}[Proof of Proposition \ref{cor1}]
Given a positive number $N$ and $c\in(0,1)$, we choose $\ep>0$ and a
3-disk $(M,g)$ as they are in Theorem \ref{thm2} such that the diameter
$d$, volume $V$ and the boundary area of the disk is bounded by six, but the
area of the cross sections of the solid tori $LT$ is greater than $N/c$. Let us modify the metric $g$ in the following way.

We first ``concentrate'' the volume of $M$ in the pair of linked solid tori
$LT$ so that if $f:D^2\ra M$ is a subdividing disk, it also subdivides
$LT$ into two regions of volume at least $c\x \vol_3(LT)$. Given any
$\eta>0$, we may choose $\de>0$ and $h:M\ra \R$ a smooth function
with $h|_{LT}=1$ and $0 < h|_{M \setminus U}\leq \de$, where $U$ is a neighbourhood of $LT$. For any $p\in M$, $u,v\in T_pM$, define $g'(u,v)=h(p)g(u,v)$. Denote by $d'$, $V'$ the diameter and volume of $(M,g')$ respectively. Then $d'\leq d$, and $V'\leq V$, because $0\leq h\leq 1$. Let $\de$ be sufficiently small so that $|\vol_3(M,g')-\vol_3(LT,g')|<\eta$. Suppose that $f:D^2\ra M$ subdivides $(M,g')$ into two regions $R_1$ and $R_2$ such that $\vol_3(R_i,g')>cV'$, for $i=1,2$. Then $\vol_3(LT\cap R_i,g')\geq cV'-\eta$.

We will then prove that in this case $f(D^2)$ must have large area. Let
$T_1$ and $T_2$ be the connected component of $LT$. Each of them is a solid torus. Without loss of generality, we may assume that the
intersection between $f(D^2)$ and the boundary $\dl LT= \dl
T_1\sqcup\dl T_2$ is transverse. Let $\bigsqcup_{j=1}^K S_j^1$ be the intersection, where each $S^1_j$ is an embedded circle in $\dl LT$.
As in the proof of Theorem~\ref{thm2}, if a circle $S^1_{j}$ is not
contractible in $\dl T_i$, for $i=1$ or $2$, then the intersection between $f(D^2)$ and $LT$ has area greater than $N/c$.

Let us now assume that all $S^1_j$ are contractible in $\dl T_i$ and consider the cylinders $C_1$ and $C_2$ ($\cong D\z [0,\ep]$) which are identified with the solid tori $T_1$ and $T_2$ respectively.
The disk $f(D^2)$ subdivides each $C_i$ into two regions.
We use the notations $R_1\cap C_i$ and $R_2\cap C_i$ to denote these regions.
Let $\rho:C_1\cup C_2\ra [0,\ep]$ be the projection onto the interval $[0,\ep]$, then the volume of $R_i\cap LT$ satisfies:

\begin{equation*}
\begin{split}
\vol_3(R_i\cap LT)&=\vol_3(R_i\cap (C_1\cup C_2))\\
&=\int_0^\ep \vol_2(\rho^{-1}(t)\cap R_i)dt\\
&\leq \ep\x\max_{t\in[0,\ep]} \vol_2(\rho^{-1}(t)\cap R_i).
\end{split}
\end{equation*}

We will now use the fact that $T_1$ and $T_2$ are linked to show the following projection equality. Note that $T_1$ and $T_2$ are identified with the cylinders $C_1$ and $C_2$. Let $p:C_1\cup C_2\ra D\cup D$ be the projection to the base disk $D$. We show that:
\begin{lemma}[The Projection Equality for Links] \label{lem: projection equality} 
	If $f : (D^2, \partial D^2) \rightarrow (D^3, \partial D^3)$ is a smooth embedding that separates $D^3$ into two closed connected regions $R_1$ and $R_2$ as above and the intersection $f(D^2)\cap \dl LT$ is contractible in $\dl LT$ then $p \left( R_i \cap LT \right) = p\left(\partial(R_i \cap LT)\right)$ for at least one of $i \in \{1,2\}$, where $LT=T_1\sqcup T_2$.
\end{lemma} 
\begin{proof}[Proof of Lemma~\ref{lem: projection equality}]
	We suppose that both equalities fail.
	In the argument that follows we will show that either $T_1$ and $T_2$ are unlinked or one of the circles in $f(D^2) \cap LT$ is non-contractible in $LT$.
	We may only work with $R_1$ and $T_1$ because our hypotheses are symmetric.
	First note that $\partial(R_1 \cap T_1) \subseteq R_1 \cap T_1$. If $p(R_1 \cap T_1) \neq p(\partial(R_1 \cap T_1))$ then we may pick $y_0 \in D$ such that $p^{-1}(y_0) \cap \partial(R_1 \cap T_1) = \emptyset$. Since $p^{-1}(y_0)\su T_1$ is connected and does not intersect $\dl (R_1\cap T_1)$, it must be contained within a single connected component of $\Int(R_1 \cap T_1)$.

	Now, consider $R_2$.
	If $p(R_2 \cap LT) \neq p(\partial(R_2 \cap LT))$ then, as before, we may pick $y \in D$ such that $p^{-1}(y) \cap \partial(R_2 \cap LT) = \emptyset$.
	By the connectedness of $p^{-1}(y)$ there are two cases, either $p^{-1}(y) \subset \Int(R_2 \cap T_1)$ or $p^{-1}(y) \subset \Int(R_2 \cap T_2)$.
	
    In the first case, suppose that $p^{-1}(y) \subset \Int(R_2 \cap T_1$).
	Note that $p^{-1}(y)$ and $p^{-1}(y_0)$ are cobordant inside $T_1$. Thus there is an annulus $A$ with $\partial A = S^1_{y} \sqcup S^1_{y_{0}}$ and an embedding $\alpha : A \rightarrow T_1$ such that $\alpha(S^1_y) = p^{-1}(y)$ and $\alpha(S^1_{y_0}) = p^{-1}(y_0)$.
	We have that $p^{-1}(y) \subset \Int(R_2)$ and $p^{-1}(y_0) \subset \Int(R_1)$.
	It follows that the annulus $\alpha(A)$ meets the disc $f(D^2)$ along a circle. The contraction of this circle in $f(D^2)$ meets an essential circle $S$ in $\partial T_1$. Because, otherwise, one can contract the circle $p^{-1}(y)$ in the solid torus $T_1$ by homotoping $p^{-1}(y)$ to $S$ and then contracting it in $\dl T_1$. This contradicts to our hypothesis that each circle in $f(D^2)\cap \dl LT$ is contractible in $\dl LT$.

    In the second case, suppose that $p^{-1}(y) \subset \Int(R_2 \cap T_2$).
	If this holds, then $p^{-1}(y)$ and $p^{-1}(y_0)$ can be separated by a disk. That is, we can contract $p^{-1}(y_0)$ in $R_1$ and $p^{-1}(y)$ in $R_2$, and thus unlink $LT$.
\end{proof}
We now complete the proof using the above lemma to obtain a lower bound for the area of $f(D^2)$. Without lost of generality we assume that Lemma~\ref{lem: projection equality} holds for the region $R_1$. Then
\begin{equation*}
\begin{split}
\ep \x \vol_2(f(D^2)\cap LT)&\geq \ep \x \vol_2\{p(f(D^2)\cap (C_1\cup C_2))\}\\
&=\ep\x \vol_2\{p(\dl (R_1\cap (C_1\cup C_2)))\}\\
&=\ep\x \vol_2\{p(R_1\cap (C_1\cup C_2))\}\\
&\geq\ep\x\max_{t\in[0,\ep]} \vol_2(\rho^{-1}(t)\cap R_1)\\
&\geq \vol_3(R_1\cap LT)>cv'-\eta,
\end{split}
\end{equation*}
Therefore, we obtain that $\vol_2(f(D^2))\geq cv'/\ep-\eta/\ep>N-\eta/\ep,$ for arbitrarily small $\eta>0$, which completes the proof.
\end{proof}

\section{Proof of Theorem~\ref{thm3}}\label{section3}
In this last section, we prove Theorem \ref{thm3}.
Given a Riemannian 3-sphere $(M,g)$ with diameter $d$ and volume $V$, consider the hypersurfaces $\Sigma$ in $M$ that subdivide $M$ into two connected components, $M \setminus \Sigma = R_1\sqcup R_2$, with both parts satisfying $\vol_3(R_i) > \frac{1}{6}V$.
We claim that for any small $\de>0$, there exists such a subdividing surface $\Sigma$ with area
	\[ \vol_2(\Sigma)\leq 3\HF_1(2d+\de)+o(\de^2)\]
By taking $\de\ra0$, we obtain the result in Theorem~\ref{thm3}.
Our argument is derived from similar filling arguments found in Nabutovsky and Rotman's work on minimal hypersurfaces~\cite{NabutovskyRotman2006}.
We prove the claim above by contradiction: we show that if there is no such surface then the fundamental class $[M] \in H_3(M;\Z)$ is zero. During the proof we will try to construct a 4-chain which has the fundamental class of $M$ as its boundary; the proof is similar to a standard coning argument.
\begin{proof}[Proof of Theorem~\ref{thm3}]
Suppose there is no subdividing hypersurface $\Sigma$ satisfying the volume bound above.
Let us choose a triangulation of $M$ such that the length of each edge of the triangulation is at most $\delta$, where $\de>0$ is a constant.
We let $\Mt = \sum \sigma_{ijkl}^3$, where each $\sigma_{ijkl}^3$ is a continuous map $\sigma_{ijkl}^3 : \Delta^3 \rightarrow M$.
Note that $\Mt$ represents the fundamental class of $M$; that is, $[\Mt] = [M] \in H^3(M;\Z)$.
We label the zero skeleton by $e_i^0 : \Delta^0 \rightarrow M$, and $v_i = e^0_i(\Delta^0)$.

We cone off inductively, skeleton by skeleton.
Let $\Mt^{(i)}$ denote the $i$-skeleton of the triangulation of $M$.
In what follows we will abuse notation and add simplices to $\Mt$ in order to complete it to a filling of $[M]$.
We begin with the zero skeleton.
Pick some generic point $v_\star$ in $M$ and add it to $\Mt^{(0)}$.
We continue to the one skeleton.
We add the following 1-chains to $\Mt$: For every $v_\alpha$ with $\alpha \neq \star$ we let $e^1_{\alpha \star} : \Delta^1 \rightarrow M$ be a minimal geodesic from $v_\alpha$ to $v_\star$.
Let $e_{\alpha\beta}^1$ denote an edge in the original triangulation $\Mt$.
Note that $\mass_1(e^1_{\alpha \beta}) \leq \delta$ and $\mass_1(e^1_{\alpha \star}) \leq d$.
We continue to the two skeleton.
We add the following 2-chains to $\Mt$: First note that for every $v_\alpha$ and $v_\beta$ adjacent in the triangulation with $\alpha, \beta \neq \star$ we have that $z^1_{\alpha \beta \star} = e^1_{\alpha\beta} + e^1_{\beta \star} + e^1_{\alpha \star}$ is a 1-cycle of mass at most $2d + \delta$.
Thus, by the definition of $\HF_1$ there is a 2-chain $\phi^2_{\alpha \beta \star}$ such that $\partial \phi^2_{\alpha \beta \star} = z^1_{\alpha \beta \star}$ and $\mass_2(\phi^2_{\alpha \beta \star}) \leq \HF_1(2d + \delta)$.
We let $e^2_{ijk}$ be the restriction of the simplex $\sigma^3_{ijkl}$ to the face containing the vertices $\{v_i, v_j, v_k\}$.
We let $z^2_{\alpha \beta \gamma \star} = e^2_{\alpha \beta \gamma} + \phi^2_{\alpha \gamma \star} + \phi^2_{\alpha \star \beta}  + \phi^2_{\gamma \beta \star}$ be the cycle we've constructed.
Note that $\mass_2(z^2_{\alpha \beta \gamma \star}) \leq 3\HF_1(2d + \delta) + o(\delta^2)$.

We now continue to the three skeleton.
In the rest of the proof, we will work in $\Z_2$-homology to avoid orientation issues.
In Lemma~\ref{lm5} we will show that, under our assumption about subdividing surfaces, each of the 2-cycles $z^2_{\alpha \beta \gamma \star}$ can be replaced by a 2-cycle $\hat{z}^2_{\alpha \beta \gamma \star}$ in $\Z_2$-coefficient with the same support which, in addition, bounds a small volume 3-chain $\hat{\phi}^3_{\alpha \beta \gamma \star}$.
In Lemma~\ref{lm5} we will show that $\mass(\hat{\phi}^3_{\alpha \beta \gamma \star}) \leq \frac{ 1 }{ 6 } V$.
Let us assume, for now, that Lemma~\ref{lm5} holds.
We will use the small volume fillings of the chains $z^2_{\alpha \beta \gamma \star}$ to fill the fundamental class of $M$.

We now continue to the four skeleton.
Consider the 3-chain formed by taking all the small volume fillings and the original 3-simplex from the triangulation of $M$:
	\[\hat{z}^3_{ijkl\star} = \hat{\phi}^3_{ijk \star} + \hat{\phi}^3_{jkl \star} + \hat{\phi}^3_{kli \star} + \hat{\phi}^3_{ijl \star} + \sigma^3_{ijkl} \]
One can check that $\partial \hat{z}^3_{ikjl \star} = 0 \in C_2(M; \Z_2)$.
By construction we have $\mass_3(\hat{z}^3_{ijkl\star}) \leq \frac{ 4 }{6 } V+o(\de^3)$.
Thus, there is a point not in the support of $\hat{z}^3_{ijkl \star}$ and we obtain that $\hat{z}^3_{ijkl \star} = \partial \hat{\phi}^4_{ijkl \star}$ for some $\hat{\phi}^4_{ijkl \star} \in C_4(M; \Z_2)$.
We then have that
\[ [M] = \left[ \sum \sigma^3_{ijkl} \right] = \left[ \sum \hat{z}^3_{ijkl\star} \right] = \left[ \partial \left( \sum \hat{\phi}^4_{ijkl \star} \right ) \right] \]
since $\sum \hat{\phi}^3_{ijk\star} = 0$.
This last equality holds since the codimension-1 faces cancel in pairs.
We now have that $[M]=0$ which is a contradiction.
\end{proof}
\begin{lemma}\label{lm5}
	Suppose that $(M,g)$ is a Riemmanian 3-sphere such that for all embedded surfaces $\Sigma \subset M$ we have the following:
		If $\mass_2(\Sigma) \leq 3\HF_1(2d)$ and $M \setminus \Sigma = R_1 \sqcup R_2$ then
		$\mass_3(R_i) \leq \frac{ 1 }{ 6 } V$ for $i=1$ or $i=2$.	
	Given such an $(M,g)$ we have that for any $z^2 = z^2_{\alpha \beta \gamma \circ}$ as above there is $\hat{z}^2 \in Z_2(M, \Z_2)$ such that:
	(i) $\supp(z^2) = \supp(\hat{z}^2)$ and (ii) there is $\hat{\phi}^3 \in C_3(M; \Z_2)$ satisfying $\partial \hat{\phi}^3 = \hat{z}^2$ and $\mass_3(\hat{\phi}^3) \leq \frac{ 1 }{ 6 } V$.
\end{lemma}

\begin{proof}

Let $z^2$ be as above,.
First note that $z^2$ is piecewise smooth by the Fleming Regularity Lemma~\cite{Fleming1962}, since it is realized as a finite union of mass-minimizing surfaces in dimension three.
We choose a sufficiently fine triangulation of the image of $z^2$ so that each simplex of the triangulation is a smoothly embedded surface $f_i$.
We note that $z^2 = \sum \ep_i f_i$ where $f_i : \Delta^2 \rightarrow M$ is smooth and $\ep_i \in \Z$.
By the $\mass_2$-minimality of $z^2$ we have that $\ep_i = \pm 1$.
Now we work with $\Z_2$-coefficients.

Consider $\hat{z}^2 = \sum f_i \in C_2(M; \Z_2)$.
We may perturb the image of $\hat{z}^2$ so that it is in general position in $M$.
That is, the image of $\hat{z}^2$ consists of:
	regular points, double arcs, triple points, and branch points \cite[Chapter 4]{carter1995surfaces}.)
We have $\mass_2(\hat{z}^2) = \mass_2(z^2)$ because $|\ep_i| = 1$ and the cycles have the same support.
By construction, $\partial \hat{z}^2 = 0$.
Thus there is $\hat{\phi}^3$ such that $\hat{z}^2 = \partial \hat{\phi}^3$ since $H_2(M; \Z_2) = 0$.

We show that there must be a small volume filling of $\hat{z}^2$.
Suppose that $\mass_3(\hat{\phi}^3) \geq \frac{ 1 }{ 6 } V$.
Let $\hat{\psi}^3$ be a chain supported in $\overline{M \setminus \supp \hat{\phi}^3}$ satisfying $\partial \hat{\psi}^3 = \hat{z}^2$.
We will prove that $\mass_3(\hat{\psi}^3) < \frac{ 1 }{ 6 } V$ by contradiction.
We will use $\hat{z}^2$, $\hat{\phi}^3$ and $\hat{\psi}^3$ to construct a subdividing surface whose area is at most $3\HF_1(2d + \delta) + o(\delta^2)$.
That is, we will show there is a surface $\Sigma$ such that: $M \setminus \Sigma = R_1 \sqcup R_2$ with
	$\vol_3(R_i) \geq \frac{ 1 }{ 6 } V$ and $\vol_2(\Sigma) \leq 3\HF_1(2d + \delta) + o(\delta^2)$.

We now construct the surface $\Sigma$.
We will first describe how to replace $\hat{z}^2$ with a union of closed embedded surfaces.
Since $\hat{z}^2$ is a piecewise smooth 2-cycle in $S^3$ we may pick an open metric ball $B(p, \eta) \subset S^3 \setminus \supp(\hat{z}^2)$ such that $S^3 \setminus B(p,\eta)$ is homeomorphic to the closed unit ball in $\R^3$.
Let $\rho$ be this homeomorphism. Then $\rho$ is $C = C(p,\eta,g)$-bilipschitz since its domain and target are both compact.

Consider image $\rho(\hat{z}^2)$ in $\R^3$.
We want to replace this cycle with a union of closed surfaces.
Let $U_{\ep} = \partial \{ x \in \R^3 : d(x, \rho(\hat{z}^2) \leq \ep\}$ be the boundary of the $\ep-$neighbourhood of $\rho(\hat{z}^2)$.
By Ferry~\cite{Ferry1976} we know that this is an embedded 2-manifold for an open dense set of $\ep \in \R^+$.
We choose $\ep$ to be sufficiently small and we pick $V_\ep$ to be a connected component of $U_{\ep}$ which deformation retracts onto the image of the cycle $\rho(\hat{z}^2)$.
Then $\hat{z}^2_\ep = \rho^{-1}(V_\ep)$ is a union of closed embedded surfaces.
Since $\rho$ is $C$-bilipschitz we can pick $\ep$ small enough so that $|\mass_2(\hat{z}^2) - \mass_2(\hat{z}^2_\ep)| \leq \ep$.

\begin{figure}[htbp]
\centering\includegraphics[width=12cm]{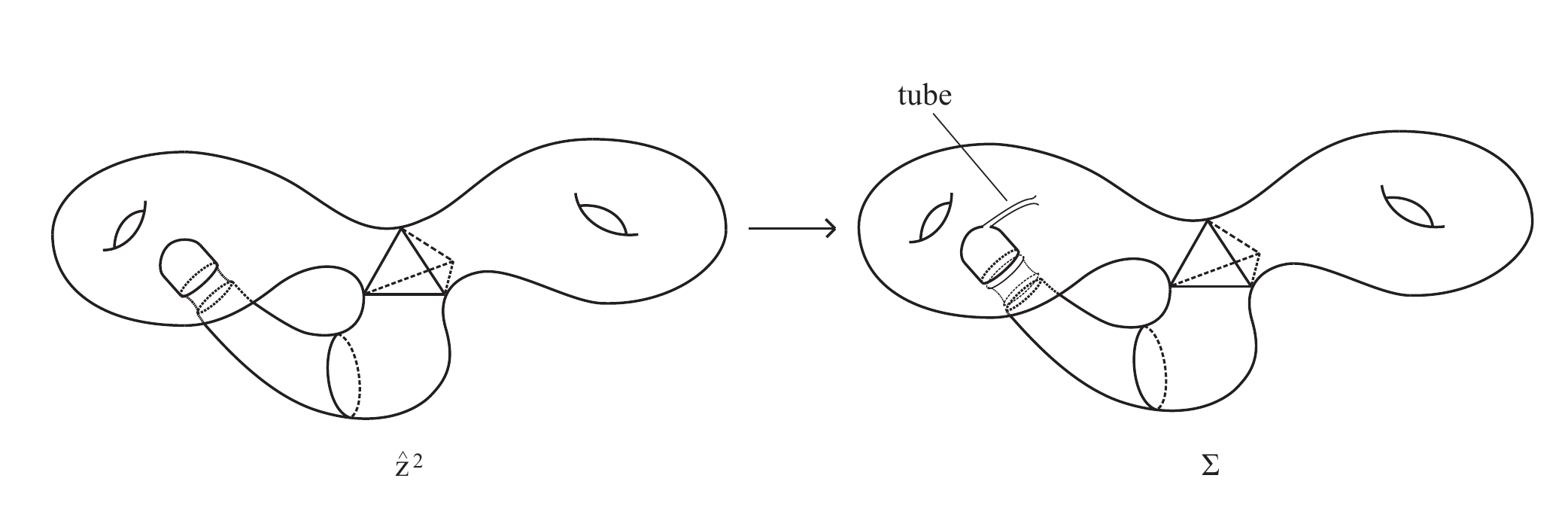}
\vspace*{8pt}
\caption{Construction of $\Sigma$. First take surfaces in a neighbourhood of $\hat{z}^2$. Then connect the components by small tubes.}\label{fig12}
\end{figure}

We now consider the surface $\Sigma$ which is obtained in the following way:
Add thin tubes to $\hat{z}^2_\ep$ so as to form one connected component. See Figure~\ref{fig12}.
We may do this while adding at most $\ep$ to both the surface area of $\hat{z}_\ep^2$, and volume of $\supp(\hat{\phi}^3)$.
We let $\Sigma$ be the boundary of $\supp(\hat{\phi^3}_\ep)$ union the thin tubes.

We have $\vol_2(\Sigma) \leq 3\HF_1(2d + \delta) + o(\delta^2) + 2\ep$, and $M \setminus \Sigma = R_1 \sqcup R_2$.
Note that $\vol_3(R_i) > \frac{ 1 }{ 6 } V - o(\ep^{\frac{ 3 }{ 2 } })$.
Taking $\ep\ra 0$ gives us a subdividing surface contradicting hypothesis about the subdivision area.
\end{proof}

\paragraph{Further Work:} Theorem~\ref{thm3} provides evidence for a positive answer to the following question:

\paragraph{Question.}
    Is there a universal constant $C$ such that any Riemannian manifold $M$ diffeomorphic to $S^3$ admits a sweep-out by 2-cycles such that the area of each cycle is bounded by $C \cdot \HF_1(2d)$?\\

A positive constructive answer to the question above would provide an effective construction of the minimal surfaces in~\cite{NabutovskyRotman2006}.
It is also desirable to have a continuous or parameterized version of Theorem~\ref{thm3}.

\section*{Acknowledgements}
The authors are grateful to their thesis advisor Regina Rotman for suggesting this problem and numerous discussions. Parker Glynn-Adey thanks Robert Young, his co-advisor, for support and patience while working on this project. We also thank Alexander Nabutovsky for discussing this work with us. This work was partially supported by NSERC grants held by Regina Rotman and Robert Young. We thank the anonymous referee for providing many useful comments.

\bigskip
\bibliographystyle{alpha}
\bibliography{mybib}

\end{document}